\nonstopmode
\documentclass[12pt,a4paper]{article}
\usepackage{amssymb,amsmath,amsthm,geometry,microtype,mathtools}\usepackage[english]{babel}
\usepackage[utf8x]{inputenc}
\usepackage[colorinlistoftodos]{todonotes}
\usepackage[colorlinks=true, allcolors=blue]{hyperref}
\author{Johan Andersson\thanks{Email:johan.andersson@oru.se \, Address:Department of Mathematics, School of Science and Technology, {\"O}rebro University, {\"O}rebro, SE-701 82 Sweden. }}
 
\title{Universality of the Hurwitz zeta-function
on the half plane of absolute convergence}
\date{}

\theoremstyle{plain} 
\newtheorem{thm}{Theorem}  
\newtheorem{lem}{Lemma}
 
\theoremstyle{definition}

\newtheorem{defn}{Definition}

\newcommand{\C}{{\mathbb C}} 
\newcommand{\R}{{\mathbb R}}
\newcommand{\N}{{\mathbb N}}
\newcommand{\Z}{{\mathbb Z}}
\newcommand{\Q}{{\mathbb Q}}

\newcommand{\abs}[1]{{\left| {#1} \right|}} 
\newcommand{\p}[1]{{\left( {#1} \right)}} 

\renewcommand{\Re}{\operatorname{Re}}

\begin{document}

\maketitle        
\begin{abstract} 
   Let $K$ be a compact set with connected complement on the half-plane $\Re(s)>0$, and let $f$ be a continuous function on $K$ which is analytic in its interior. We prove that for any parameter $0<\alpha<1, \alpha \neq \frac 1 2$  then $f(s)$ may be uniformly approximated arbitrarily closely by $\zeta(1+iT+i\delta s,\alpha)$ on $K$ for some $T,\delta>0$, where $\zeta(s,\alpha)$ denote the Hurwitz zeta-function. This is the first known universality result that is also known to hold for the Hurwitz zeta-function with an algebraic irrational parameter.
\end{abstract}

\maketitle
\tableofcontents
\section{Introduction and main results}
In a recent paper \cite{Andersson1} we introduced a new idea on how to prove a new type of universality theorem for a Dirichlet series with an Euler-product. By using an elaboration of this idea we obtain joint universality results for Dirichlet  $L$-functions that implies the corresponding theorem for the Hurwitz zeta-function with a rational parameter. We also manage to apply the same general idea to prove universality results of the Hurwitz zeta-function   with an irrational parameter.  A great advantage with our approach is that in contrast to previous universality theorems it also works for algebraic irrational parameters. To prove universality (in the classical sense) for the Hurwitz zeta-function with an algebraic irrational parameter is a well-known open problem in the field. For a discussion of some of the difficulties, see \cite{Andersson4}, \cite[pp. 17-18]{Andersson5}. The only known results in this direction are some very recent results of  Sourmelidis-Steuding, see e.g. \cite[Theorem 2]{SoSte} where unfortunately to obtain  approximations, we do not only have to vary the imaginary shift $t$, but also the algebraic parameter $\alpha$.  For a zero-distribution result in this direction see  Garunk\v stis \cite{Garunkstis}. There are also some related limit distribution results in the literature \cite{LaSte1}, \cite{La}. Our main result is the following theorem.
   \begin{thm} \label{th} Let $0 < \alpha<1$, with  $\alpha \neq \frac 1 2$. Let
  $K \subset \C$ be a compact set with connected complement, where $K \subset \{s  \in \C:\Re(s)>0\}$ if $\alpha$ is algebraic irrational, and suppose that $\varepsilon>0$ and that $f$ is any continuous function on $K$ that is analytic  in the interior of $K$.  Then there exist some $\delta_0>0$  such that for any  $0 <\delta \leq \delta_0$ then
\[
 \liminf_{T \to \infty} \frac 1 T \mathop{\rm meas} \left \{t \in [0,T]:\max_{s \in K} \abs{\zeta(1+it+\delta s,\alpha)-f(s)}<\varepsilon \right \}>0. 
\]
\end{thm} 
The Voronin universality theorem for the Hurwitz zeta-function\footnote{Proved independently by Bagchi \cite{Bagchi} and Gonek \cite{Gonek}, generalizing the corresponding result for the Riemann zeta-function \cite{Voronin}.}  requires that $K \subset \{s \in \C: -\frac 1 2<\Re(s)<0\}$ in Theorem \ref{th} but allows us to choose $\delta=1$. The reason why our method works also for algebraic irrational numbers is because in our case we work on the half plane of absolute convergence which is easier to understand than the critical strip, and the approach of Cassels \cite{Cassels} 
applies. Since we may allow $f$ in Theorem \ref{th} to have zeros on the set $K$, by Rouche's theorem  an immediate consequence is  that the Hurwitz zeta-function $\zeta(s,\alpha)$ for parameters $0<\alpha<1,  \alpha \neq \frac 1 2$ contains infinitely many zeroes in each strip $1<\Re(s)<1+\delta$. Thus as a consequence of Theorem \ref{th} we obtain classical results of Davenport-Heilbronn \cite{DavHeil} and Cassels \cite{Cassels}.

We would like to remark that in order for  Theorem \ref{th} to be true, we must allow the scaling factor $\delta$, since Theorem \ref{th} is not true without it. This was proved in \cite[Corollary 3]{Andersson99} and in a more explicit form in \cite[Theorem 29]{Andersson6}
\begin{gather} \label{i2i}
  \delta^{7/{6\delta}} 10^{-9/\delta} \leq \inf_{0 <\alpha \leq 1}  \inf_T  \int_T^{T+\delta} \left| \zeta(1+it,\alpha) \right| dt, \qquad (0< \delta \leq 0.05).
\end{gather}
This follows since \eqref{i2i} does not allow the function $\zeta(1+it+iT,\alpha)$ to approximate the function $f(t)=0$  arbitrarily closely on any interval $[0,\delta]$ for fixed $\delta$.

In the case where the Hurwitz zeta-function has an Euler-product, for the parameters $\alpha=\frac 1 2$ and $\alpha=1$ our corresponding theorem will be somewhat weaker as we need to add some unspecified constant $C$. The following is a consequence of  
\cite[Theorem 1]{Andersson1}.
  \begin{thm} \label{th2} Suppose that $\alpha=\frac 1 2$ or $\alpha=1$. Let
  $K \subset \C$ be a compact set with connected complement, and suppose that $\varepsilon>0$ and that $f$ is any continuous function on $K$ that is analytic  in the interior of $K$.  Then there exist some $\delta_0, C_0>0$  such that for any $0<\delta \leq \delta_0$ and $|C|>C_0$ then
\[
 \liminf_{T \to \infty} \frac 1 T \mathop{\rm meas} \left \{t \in [0,T]:\max_{s \in K} \abs{\zeta(1+it+\delta s,\alpha)+C-f(s)}<\varepsilon \right \}>0. 
\]
\end{thm} 
\begin{proof} The case $\alpha=1$ is exactly \cite[Theorem 1]{Andersson1}. For  $\alpha=\frac 1 2$ we have that
\begin{gather*}
  \zeta \p{s,\frac 1 2}=(2^s-1) \zeta(s).
\end{gather*} and it follows from a corresponding hybrid universality version of     \cite[Theorem 1]{Andersson1} where we may also assume that $|2^{it}+1|<\varepsilon$. This hybrid universality result follows easily  by the same proof method as \cite[Theorem 1]{Andersson1} since we may choose $a_2=-1$ in \cite[Lemma 2]{Andersson1}. Indeed we will do the details in this paper and Theorem \ref{th2} is a special case of Theorem \ref{th4} which is stated on the next page for a single Dirichlet $L$-function with the principal character mod $2$, and with $a_2=-1$. For details on how to prove an even  more general hybrid universality result, see the proof of \cite[Theorem 3]{Andersson3}.
\end{proof}
Next we state a joint hybrid universality theorem generalizing  \cite[Theorem 2]{Andersson1} which implies Theorem  \ref{th} for the case of a rational parameter $\alpha$.  
\begin{thm} \label{th3} Let $|a_p|=1$ be given for the primes $p \leq N$. Let  $\chi_1,\ldots,\chi_n$ be pairwise non-equivalent Dirichlet characters,  let $K \subset  \C$ be a compact set and let $f_1,\ldots,f_n$ be functions such that
 \begin{gather*}
   f_k(s)= C_k+\int_0^{\infty} g_k(x)e^{-sx} dx,  \qquad (s \in K) 
 \\ \intertext{where}
    \sum_{k=1}^n \abs{xg_k(x)} \leq 1, \qquad (x \geq 0). 
       \end{gather*}
       Then for any  $\varepsilon>0$ there exist some  $\delta_0>0$  such that for each $0<\delta\leq \delta_0$  then
       \begin{gather*}
            \liminf_{T \to \infty} \frac 1 T \mathop{\rm meas} \left \{t \in [0,T]:\begin{matrix*}[l]   \max_{p \leq N} \abs{p^{-it}-a_p}<\varepsilon, \\ 
            \max_ {1 \leq k \leq n} \max_{s \in K} \abs{\log L(1+it+\delta s,\chi_k)-f_k(s)}<\varepsilon  \end{matrix*} 
            \right \}>0. 
 \end{gather*} \end{thm}

   The condition on the functions $f_k$ in Theorem \ref{th3} is somewhat restricted and similarly to \cite[Theorem 1]{Andersson1} we may remove this condition if we are also allowed to add some sufficiently large constant.
\begin{thm} \label{th4}  Let $|a_p|=1$ be given for the primes $p \leq N$. Let $\chi_1,\ldots,\chi_n$ be pairwise non-equivalent Dirichlet characters and, let $K \subset  \C$ be a compact set with connected complement and let $f_1,\ldots,f_n$ be continuous  functions on $K$ that are analytic in its interior.  Then for any $\varepsilon>0$ there exists some $C_0,\delta_0>0$, such that for any $|C_k| \geq C_0$ and  $0<\delta\leq \delta_0$ then
       \begin{gather*}
            \liminf_{T \to \infty} \frac 1 T \mathop{\rm meas} \left \{t \in [0,T]:\begin{matrix*}[l]   \max_{p \leq N} \abs{p^{-it}-a_p}<\varepsilon, \\  
            \max_ {1 \leq k \leq n} \max_{s \in K} \abs{ L(1+it+\delta s,\chi_k)+C_k-f_k(s)}<\varepsilon
            \end{matrix*} 
             \right \}>0. 
 \end{gather*} \end{thm}
 Finally, for the Lerch zeta-function
 $$
  L(\lambda,\alpha,s)=\sum_{k=0}^\infty \frac{\exp\p{2 \pi \lambda i k}}{(k+\alpha)^s}
 $$
  we obtain a result corresponding to Theorem \ref{th}. For simplicity we state it only of $\alpha$ is irrational. In general we can also handle the case of both $\alpha$ and $\lambda$ rational by writing the Lerch zeta-function as a combination of Dirichlet $L$-functions see \cite[Section 6.2] {GaLa}. However, in the same way as in the classical setting, we do not know how to prove  universality for the case when $\alpha$ is rational and $\lambda$ is irrational.
   \begin{thm} \label{th5} Let $0<\alpha \not \in \Q$. Let
  $K \subset \{s  \in \C:\Re(s)>0\}$ be a compact set with connected complement and suppose that $\varepsilon>0$ and that $f$ is any continuous function on $K$ that is analytic  in the interior of $K$.  Then there exist some $\delta_0>0$  such that for any  $0 <\delta \leq \delta_0$ then
\[
 \liminf_{T \to \infty} \frac 1 T \mathop{\rm meas} \left \{t \in [0,T]:\max_{s \in K} \abs{L(\lambda,\alpha,1+it+\delta s)-f(s)}<\varepsilon \right \}>0. 
\]
\end{thm}

 \section{Proof of joint universality theorems}
 
 Let us first prove our joint universality results, Theorem \ref{th3} and Theorem \ref{th4}.  The following Lemma (for its proof, see section \ref{sec4}) which is a joint universality variant of \cite[Lemma 1]{Andersson1} replaces the Pechersky rearrangement theorem in the classical universality proof.
 \begin{lem}
 \label{Le2}
 For any  $|a_p|=1$ for $p \leq N$,   $\varepsilon>0$, compact set $K$,  characters $\chi_1,\ldots,\chi_n$  and functions  $f_1,\ldots,f_n$,  satisfying the conditions of Theorem \ref{th3}  there exists some $\delta_0>0$ such that for any $0<\delta \leq \delta_0$  there are unimodular complex numbers $|\omega(p)|=1$   such that $\omega(p)=a_p$ for $p \leq N$ and 
  \begin{gather*}
     h_k(s)=
-\sum_{p} \log \p{1-\frac{\chi_k(p) \omega(p)} {p^{1+\delta s}}},   \\ \intertext{are convergent to analytic functions $h_k$ for $\Re(s)>\frac 1 2$ and such that}
     \max_{1\leq k \leq n} \max_{s \in K} \abs{h_k(1+\delta s)-f_k(s) }<\varepsilon.
  \end{gather*}
 \end{lem}
 In order to prove these results we also need some  well-known fact from the theory of universality, which we state in the following convenient form. 
\begin{lem} \label{LALA}
  Let \begin{gather*}  h_k(s)=-\sum_{p}  \log \p{1-\frac{\omega(p) \chi_k(p)}{ p^{s}}} \end{gather*} where $|\omega(p)|=1$, $\chi_k$ are Dirichlet characters where the series for each $1\leq k \leq n$ are convergent to  analytic functions $h_k$ on the half-plane $\Re(s)>\frac 1 2$. Then for any $\varepsilon>0$, $N \in \Z^+$ and compact set $K \subset \{s \in \C: \Re(s)>\frac 1 2\}$ we have that
  \begin{gather*}
    \liminf_{T \to \infty} \frac 1 T \mathop{\rm meas} \left \{t \in [0,T] : \begin{matrix*}[l]   \max_{p \leq N} \abs{p^{-it}-\omega(p)}<\varepsilon, \\ \max_{1 \leq k \leq n} \max_{s \in K} \abs{\log L(s+it,\chi_k)-h_k(s) } <\varepsilon \end{matrix*} \right \}>0. 
\end{gather*}
\end{lem}
\begin{proof}
We have not found this precise statement of the Lemma in the literature, but it is likely somewhere (should look further). It is clear though that its proof is standard, for example it follows from methods of \cite[Section 12.1]{Steuding}.
\end{proof} 
\vskip 8pt

\noindent {\em Proof of Theorem \ref{th3}.}
 By Lemma \ref{Le2} we may find some  $\delta_0>0$ such that for any $0<\delta \leq \delta_0$ there exists some series
\begin{gather*} h_k(s)=-\sum_p \log \p{1-\frac{\omega(p) \chi_k(p)}{ p^{s}}},
\end{gather*} such that \begin{gather} \label{omp} \omega(p)=a_p, \qquad p \leq N,
\end{gather}
and $|\omega(p)|=1$ which for each $k=1,\ldots,n$ are convergent on the half-plane $\Re(s)>\frac 1 2$ to analytic functions $h_k$ 
such that
\begin{gather} \label{uu1}
   \max_{1\leq k \leq n} \max_{s \in K} \abs{h_k(1+\delta s)-f_k(s)} <\frac {\varepsilon} 2.
\end{gather}
By using Lemma \ref{LALA} with the compact set $1+\delta K$ which for a sufficiently small $\delta$ lies in the half plane $\Re(s)>\frac 1 2$, it follows that
  \begin{gather} \label{uu2}
 \liminf_{T \to \infty} \frac 1 T \mathop{\rm meas} \left \{t \in [0,T]:\begin{matrix*}[l]   \max_{p \leq N} \abs{p^{-it}-\omega(p)}<\varepsilon, \\  \max_{1\leq k \leq n} \max_{z \in 1+\delta K} \abs{\log L(z+it,\chi_k)-h_k(z)}<\frac \varepsilon  2   \end{matrix*}
 \right \}>0. 
\end{gather}
Our result follows by the change of variable $z=1+\delta s$, the identity \eqref{omp}, the inequalities \eqref{uu1}, \eqref{uu2} and the triangle inequality.
 \qed
 \vskip 8pt
 
 The next lemma appears as
\cite[Lemma 3]{Andersson1} and is a consequence of Mergelyan's theorem and the theory of Laplace transforms. 
\begin{lem} \label{LA3} Assume that $f$ is any zero-free function on a compact set $K$ with connected complement analytic in its interior. Then given $\varepsilon>0$ there exist some $A,B,N>0$  and continuous function $g:[A,B] \to \C$ such that $|g(x)| \leq N$ and that if
\begin{gather*}
  G(s) = \int_A^B g(x) e^{-sx} dx, \\ \intertext{then}
  \max_{s \in K} \abs{G(s)-f(s)}<\varepsilon,
\end{gather*}
\end{lem}
\begin{proof} See \cite[Lemma 3]{Andersson1}. \end{proof}

\noindent {\em Proof of Theorem \ref{th4}.} 
 It is sufficient to show that for any $\varepsilon_1>0$ there exist some $\delta_0>0$ and $C_0>0$ such that for any $|C_k| \geq C_0$ then
\begin{gather} \label{ab333}
    \max_{1 \leq p \leq N} \abs{p^{-it}-a_p}<\varepsilon,  \\ \intertext{and}   \label{ab3}
       \max_{1 \leq k \leq n} \max_{s   \in K}     \abs{\log(L(1+it+\delta s,\chi_k))-\log(f_k(s)-C_k)}<\frac{\varepsilon_1}{|C_k|}, 
          \end{gather}
holds with a positive lower measure $0\leq t \leq T$ as $T \to \infty$. It is clear that
\begin{gather} \label{iden2}
   \log(f_k(s)-C_k)=\log \p{-C_k \p{1- \frac {f_k(s)}  {C_k}}}= \log (-C_k)+\log \p{1-\frac {f(s)}   {C_k}},
\end{gather}
and if we  choose  $|C_k| \geq 1+ \max_{
s \in K} {|f_k(s)|}(4 \varepsilon_1^{-1})$ then 
\begin{gather} \label{ab1}
 \abs{\frac{f_k(s)} {C_k} +\log \p{1-\frac {f_k(s)} {C_k}}}<\frac{\varepsilon_1} {3|C_k|}. 
\end{gather}
By Lemma \ref{LA3} there exist some $0<A<B$ and continuous functions  $g_k:[A,B] \to \C$ such that
\begin{gather} \label{ab0} 
 \max_{s \in K}  \abs{G_k(s)-    f_k(s)}<\frac{\varepsilon_1} 3,  \\ \intertext{where} \notag
    G_k(s)=\int_A^B g_k(x) e^{-sx} dx.
\end{gather}
If we also choose $|C_k| \geq  n\max_{A \leq x \leq B} |x g_k(x)|$ then the
functions $$h_k(s)=\frac {G_k(s)} {C_k}+  \log {(-C_k)}$$ satisfy the condition of Theorem  \ref{th3} so that there exists some $\delta_0 >0$  where \eqref{ab333} and 
\begin{gather} \label{ab2}
  \max_{s \in K} \abs{\log L(1+it+\delta s,\chi_k)-\frac{G_k(s)} {C_k}- \log (-C_k)}<\frac {\varepsilon_1} {3|C_k|}
\end{gather}
holds for any  $0<\delta \leq \delta_0$  with a  positive lower measure $0\leq t \leq T$ as $T \to \infty$. Finally the inequality \eqref{ab3} follows by  the identity \eqref{iden2}, the inequalities \eqref{ab1}, \eqref{ab0}, \eqref{ab2} and the triangle inequality.  \qed

\section{Proof of universality of the Hurwitz zeta-function}

In this section we prove our universality theorem for the Hurwitz zeta-function, Theorem \ref{th}. For the rational case we will use Theorem \ref{th4}. For the irrational case we first introduce some notation
\begin{defn}
We say that $\omega:\{\alpha,\ldots,\alpha+N\} \to \C$ is a completely multiplicative unimodular function if $|\omega(n+\alpha)|=1$ for $n=0,\ldots,N$ and if
\begin{gather}
\label{cond2}   \prod_{n=0}^N (n+\alpha)^{b_n}=1, \qquad (b_n \in \Z) \\ \intertext{then}
   \label{cond3} \prod_{n=0}^N \p{\omega(n+\alpha)}^{b_n}=1. 
\end{gather}
\end{defn}
We need the following Lemma (for its proof, see section \ref{sec5}) 
 which is a substitute for the Pechersky rearrangement theorem in the classical arguments.  
\begin{lem} \label{LA}
 For any irrational parameter $\alpha>0$,  compact set $K$, $\varepsilon>0$  and function  $f$,  satisfying the conditions of Theorem \ref{th} there exists   some  $\delta_0>0$ such that for any $0<\delta \leq \delta_0$ there exists  some $N_0$ such that for any $N \geq N_0$ there exists some completely multiplicative unimodular function $\omega:\{\alpha,\ldots,N+\alpha\} \to \C$ such that
 \begin{gather} \notag
    \max_{s \in K} \abs{\sum_{n=0}^N  \frac{\omega(n+\alpha)} {(n+\alpha)^{1+\delta s}} -f(s)} <\varepsilon.
\end{gather}
 \end{lem}
 For the case when $\alpha$ is transcendental we do not require that $K$ lies in the half plane $\Re(s)>0$ and we need the following standard limit theorem for the Hurwitz zeta-function
 \begin{lem} \label{LALA5}
  Let $\alpha$ be transcendental and let \begin{gather*}  H(s)=\sum_{n=0}^\infty \frac{\omega(n+\alpha)}{(n+\alpha)^s} \end{gather*} such that $|\omega(n+\alpha)|=1$ be a Dirichlet series convergent on the half plane $\Re(s)>\frac 1 2$.  Then for any $\varepsilon>0$ and compact set $K \subset \{s \in \C: \Re(s)>\frac 1 2\}$ we have that
  \begin{gather*}
 \liminf_{T \to \infty} \frac 1 T \mathop{\rm meas} \left \{t \in [0,T]: \max_{s \in K} \abs{\zeta(s+it,\alpha)-H(s)}<\varepsilon \right \}>0. 
\end{gather*}
\end{lem}
\begin{proof}
We have not found this precise statement of the Lemma in the literature, but it is likely somewhere (should look further). It is clear though that its proof is standard.
\end{proof}

\noindent {\em Proof of Theorem \ref{th}.}

We divide the proof according to whether $\alpha$ is rational, or irrational and divide the irrational case according to whether $\alpha$ is algebraic or transcendental
\begin{description}
  \item[$\alpha$ rational:] 
  Let 
  \begin{gather} \notag \alpha= \frac p q, \qquad (p,q)=1, \\ \intertext{where $q \geq 3$. Then} \zeta(s,\alpha)= \frac{q^{s}} {\varphi(q)} \sum_{\chi \pmod q}  
   \overline{\chi(p)} L(s,\chi).  \label{iden} 
  \end{gather}
   By Theorem \ref{th3} we may find some sufficiently large $C$ such that for any $0<\delta<\delta_0$ we may approximate
   \begin{gather}  \label{ineq}
      \max_{s \in K} \max_{\chi \pmod q} \abs{L(1+it+ \delta s,\chi) - C- \frac{\chi(p)} {q} f(s)}<\frac{\varepsilon}{2 q},  \\ \intertext{and} \label{ineq9}  \abs{q^{it}-1}<\varepsilon_2,
\end{gather}   
on a set of positive lower measure  in $t$. By the inequality \eqref{ineq} and the identities  \eqref{iden},  
\begin{gather*} 
  \sum_{\chi \pmod q} \chi(p) =0,
\end{gather*}
 and the triangle inequality this implies that
\begin{gather} \label{aa}
   \max_{s \in K} \abs{q^{\delta s+it}  \zeta(1+it+\delta s,\alpha)-f(s)}< \frac {\varepsilon} 2
\end{gather}
on a set of positive lower measure  in $t$. If  $\delta$ and $\varepsilon_2$ are  chosen sufficiently small then it furthermore follows from \eqref{ineq9} and the fact that $K$ is compact that  $q^{\delta s+it}$ is sufficiently close to $1$ such that 
\begin{gather} \label{ab}
   \max_{s \in K} \abs{q^{\delta s+it}  \zeta(1+it+\delta s,\alpha) -  \zeta(1+it+\delta s,\alpha)}< \frac{\varepsilon} 2
\end{gather} 
for these values of $t$.  The result follows from the triangle inequality and the inequalities \eqref{aa} and \eqref{ab}.
   \item[$\alpha$ irrational:] 
        By Lemma \ref{LA} there exists some $\delta_0$ such that for any $0<\delta\leq \delta_0$  and any sufficiently large $N$  there exists a completely multiplicative unimodular function $\omega:\{\alpha,\ldots,N+\alpha\} \to \C$   such that 
     \begin{gather} \label{tra1}
     \max_{s \in K} \abs{\sum_{n=0}^N  \frac{\omega(n+\alpha)} {(n+\alpha)^{1+\delta s}} -f(s)} <\frac{\varepsilon} 3.
    \end{gather}
    \begin{description} \item[$\alpha$ algebraic:]
     In the algebraic irrational case we assume that $N$ in \eqref{tra1} is sufficiently large so that if $\xi:=\min_{s \in K}\Re(s)>0$ then
     \begin{gather} \label{tra2}
       \sum_{k=N+1}^\infty \frac 1 {(k+\alpha)^{1+\delta \xi}} <\frac \varepsilon 3.
      \end{gather}
     Now let $\{\log(n_1+\alpha),\ldots,\log(n_k+\alpha)\}$ be a basis for $\operatorname{Span}_\Q \{\log(n+\alpha):0 \leq n \leq N\}$. It is clear that if
     \begin{gather} \label{ire5}
          \max_{1 \leq j \leq k} \abs{(n_j+\alpha)^{-it}-\omega(n_j+\alpha)}<\varepsilon_2 \end{gather}
      for some sufficiently small $\varepsilon_2>0$  then 
      \begin{gather} \label{tra3}
       \max_{s \in K}  \abs{\sum_{n=0}^N   \frac{\omega(n+\alpha)} {(n+\alpha)^{1+\delta s}} -\sum_{n=0}^N  \frac{1} {(n+\alpha)^{1+it+\delta s}} } <\frac{\varepsilon} 3.
      \end{gather}
It follows by \eqref{tra1}, \eqref{tra2}, \eqref{tra3} and the triangle inequality that
\begin{gather*}
  \max_{s \in K} \abs{ \zeta(1+it+\delta s,\alpha)-f(s)} <\varepsilon,
\end{gather*} 
whenever $t$ satisfies the inequality \eqref{ire5} for some sufficiently small $\varepsilon_2>0$. By Weyl's version (see \cite{Weyl} or  \cite[Lemma 1.8]{Steuding}) of the Kroenecker's approximation theorem the set of $0 \leq t \leq T$ where \eqref{ire5} holds has positive lower density as $T\to \infty$.
\item[$\alpha$ transcendental:] In the transcendental case we assume that $N$ in \eqref{tra1} is  chosen sufficiently large so that
\begin{gather} \label{tra4}
 \max_{s \in K} \abs{\sum_{n=N+1}^\infty \frac {(-1)^n}{(n+\alpha)^{1+\delta s}}}<\frac{\varepsilon} 3.
\end{gather}
We now extend $\omega$ to be a unimodular function $\omega:\N+\alpha \to \C$ by defining $\omega(n+\alpha)=(-1)^n$ for $n \geq N+1$. By assuming that $\delta_0$ is sufficiently small such that $1+\delta_0 K \subset \{s \in \C: \Re(s)>\frac 1 2\}$ it  follows from Lemma \ref{LALA5} that
\begin{gather} \label{tra5}
 \liminf_{T \to \infty} \frac 1 T \mathop{\rm meas} \left \{t \in [0,T]: \max_{z \in 1+ \delta K} \abs{\zeta(z+it,\alpha)-\sum_{n=0}^\infty \frac {\omega(n+\alpha)}{(n+\alpha)^{z}}}<\frac \varepsilon 3 \right \}>0. 
\end{gather}
The conclusion of the transcendental case follows from the change of variables $z=1+\delta s$, the inequalities  \eqref{tra1}, \eqref{tra4}, \eqref{tra5} and the triangle inequality.
\end{description}
   \end{description} \qed

\section{The rational case - Proof of Lemma \ref{Le2} \label{sec4}}

 \begin{lem} \label{lem5}
    Let $\chi_1,\ldots,\chi_n$ be pairwise non-equivalent Dirichlet characters.
    Then for any $\xi>0$ and $\varepsilon>0$ there exists some $P_0>0$ such that if  $b_k$ satisfies  \begin{gather} \label{hury}
    \sum_{k=1}^n \abs{b_k}  \leq 1,
       \end{gather}
       and $P>P_0$, then  there exists some unimodular numbers $|\omega(p)|=1$ for primes $P<p \leq P^{1+\xi}$  such that
       \begin{gather*}
          \max_{1\leq k \leq n}\abs{\frac 1 {\log(1+\xi)} \sum_{P<p \leq P^{1+\xi}}\frac{\omega(p) \chi_k(p)} p -b_k} <\varepsilon.
\end{gather*}
\end{lem}
\begin{proof} Define
$$
b_{n+1}:=1-\sum_{k=1}^n \abs{b_k}.
$$The condition \eqref{hury} allows us to divide the interval $$[1,1+\xi] = \bigcup_{j=1}^{n+1}[c_{j-1},c_j]$$  into subintervals such that $c_0=1$ and  
\begin{gather} \label{ere2} \log c_{k}- \log c_{k-1} = \log(1+\xi)\, \abs{b_k}, \qquad (k=1,\ldots,n+1).
\end{gather}
We now let $\chi_{n+1}$ be a Dirichlet character  pairwise non-equivalent to each $\chi_k$ for $k=1,\ldots,n$ and define
\begin{gather} \label{ere2b}
  \omega(p)=\overline{\chi_k(p)}\begin{dcases} \frac{b_k}{\abs{b_k}}, & b_k \neq 0, \\ 1, & b_k=0, \end{dcases}  \qquad  (P^{c_{k-1}} < p \leq  P^{c_k}),
\end{gather}
for each $k=1,\ldots,n+1$. From the prime number theorem for arithmetic progressions we have that
\begin{gather*}
   \sum_{p \leq P} \frac{\chi_k(p) \overline{\chi_j(p)}} p =\delta_{k,j} \log \log P+ C_{k,j}+o(1), \qquad \delta_{k,j}=\begin{cases} 1, & k=j, \\ 0, & \text{otherwise,} \end{cases} 
\end{gather*}
from which it follows together with \eqref{ere2} and \eqref{ere2b} that 
\begin{gather} \label{aj2}
  \abs{\frac 1 {\log(1+\xi)} \sum_{P^{c_{k-1}} < p \leq  P^{c_k}} \frac{\omega(p) \chi_j(p)}p -\delta_{k,j} b_k }<\frac{\varepsilon} {n+1},
\end{gather} 
provided $P$ has been chosen sufficiently large. Our conclusion follows for each $j=1,\ldots,n$ from the inequality \eqref{aj2} for $k=1,\ldots,n+1$ and the triangle inequality.
\end{proof}

\begin{lem} \label{lem6}
  Let $\chi_1,\ldots,\chi_n$ be pairwise non-equivalent Dirichlet characters.   Given any unimodular numbers $|a_p|=1$ for $p \leq N$ and complex constants $C_1,\ldots,C_n$ and $\varepsilon>0$ there exists some $P_0$ such that for any $P \geq P_0$ we can find some unimodular numbers $|\omega(p)|=1$ for primes $p \leq P$ such that $\omega(p)=a_p$ for $p \leq N$ and 
  \begin{gather*}
    \max_{1\leq k \leq n}\abs{ \sum_{p \leq P}\log \p{1-\frac{\omega(p) \chi_k(p)} p} + C_k} <\varepsilon.
  \end{gather*}
  \end{lem}
\begin{proof} 
   Let $\chi$ be a Dirichlet character pairwise non-equivalent to $\chi_k$ for each $1 \leq k \leq n$. From the prime number theorem for arithmetic progressions we have
     \begin{gather*}
   -\sum_{p<P} \log\p{1-\frac{\chi_k(p) \chi(p)} p} =D_k+o(1).
   \end{gather*}
   Let us now define 
   $$
    E_k:=\frac 1 M \p{C_k+D_k+\sum_{p \leq N}\p{ \log\p{1-\frac{\chi_k(p) \chi(p)} p} -  \log\p{1-\frac{\chi_k(p)a_p} p}}},
   $$
   where $M \in \Z^+$ is sufficiently large so that $|E_k|<\log 2$ for $k=1,\ldots,n$. Let  $P_1$ be sufficiently large such  that if $P\geq P_1$ then
      \begin{gather}\label{uri}
   \max_{1\leq k \leq n} \abs{\sum_{p \leq P} \log \p{1-\frac{\chi_k(p) \chi(p)} p} +D_k}<\frac {\varepsilon} 3,  \\ \intertext{and}
      \max_{|\omega(p)|=1} \max_{1\leq k \leq n} \sum_{p>P_1} \abs{\log \p{1-\frac{\omega(p)\chi_k(p)} p}   + \frac{\omega(p)\chi_k(p)} p }<\frac{\varepsilon} 3, \label{uri2}
      \end{gather}
   and such that for each $Q \geq P_1$ we can use Lemma \ref{lem5} with $\xi=1$  to define $|\omega(p)|=1$ for $Q<p \leq Q^2$ such that
   \begin{gather} \label{ure}
       \max_{1\leq k \leq n}\abs{\sum_{Q<p \leq Q^2}\frac{\omega(p) \chi_k(p)} p -E_k} <\frac{\varepsilon}{3M}.
   \end{gather}
   By defining $\omega(p)=a_p$ for $p \leq N$ and $\omega(p)=\chi(p)$ for $N<p \leq Q$ and by \eqref{ure} for $Q<p \leq Q^2$ for  $Q=P^{2^j}$ for each $j=0,\ldots,M-1$ then the conclusion of the lemma follows with $P_0=P_1^{2^M}$ by the inequalities \eqref{uri}, \eqref{uri2}, \eqref{ure} and the triangle inequality. 
    \end{proof}

 \noindent {\em Proof of Lemma \ref{Le2}.} 
 Choose\footnote{The convergence is well-known for ``almost all'' $\omega_0$ in a suitable sense. By Carleson's theorem we may even choose $\omega_0(p)=e^{2 \pi i p x}$ for almost all $0\leq x \leq 1$} $|\omega_0(p)|=1$ such that the Dirichlet series
 \begin{gather*}
   A_k(s)=\sum_p \frac{\omega_0(p) \chi_k(p)} {p^s}
 \end{gather*}  
  are convergent to analytic functions on the half plane $\Re(s)>\frac 1 2$ for each $k=1,\ldots,n$. Choose $P_0$ and $\delta_1$ sufficiently small such that $1+\delta_1 K \subset \{s \in \C: \Re(s)>3/4\}$ and such that
  \begin{gather} \label{oo3}
     \sup_{P \geq P_0} \max_{0 \leq \delta \leq \delta_1} \max_{s \in K} \sum_{s \in K} \abs{\sum_{p > P} \frac{\omega_0(p) \chi_k(p)}{p^{1+\delta s}}}<\frac{\varepsilon} 9.
  \end{gather}
By Lemma \ref{lem6} we may find some $P_1\geq P_0$ and $|\omega(p)|=1$ for $p<P_1$ such that
\begin{gather} \label{iia0}
\max_{1 \leq k \leq n}\abs{ 
\sum_{p \leq P_1} \log\p{1-\frac{\omega(p)\chi_k(p)}{ p}} +C_k}  <\frac{\varepsilon} 9, 
\end{gather}
and such that \begin{gather}
 \label{ajaj}  \max_{s \in K} \max_{1 \leq k \leq n} \max_{|\omega(p)|=1}
\sum_{p > P_1}  \abs{\log\p{1-\frac{\omega(p)\chi_k(p)}{ p}}  + \frac{\omega(p)\chi_k(p)}{ p} }  
  <\frac{\varepsilon} 9,
\end{gather}
Let us now choose $B>0$ such that
\begin{gather} \label{ia0}
\max_{1 \leq k \leq n} \max_{s \in K} \abs{\int_0^B g_k(x) e^{-sx} dx-f_k(s)}<\frac \varepsilon 9,
\end{gather}
and  $M \in \Z^+$  sufficiently large 
such that
\begin{gather} \label{in1}
 \max_{1 \leq k \leq n} \max_{\substack{|x-y| \leq B/M \\ 0  \leq x,y \leq B}} \max_{s \in K} \abs{e^{-sx} g_k(x)-e^{-sy}g_k(y)} <\frac{\varepsilon} {9B}.
 \end{gather}  
From  \eqref{in1} it follows that the Riemann integral may be estimated by a Riemann sum
\begin{gather} \label{ia1}
\max_{1 \leq k \leq n} \max_{s \in \C} \abs{\sum_{m=1}^M g\p{\frac{mB}M} e^{-smB/M} \frac B M   - \int_0^B g_k(x) e^{-sx} dx}<\frac \varepsilon 9.
  \end{gather}
For a given $\delta$, define 
\begin{gather} \label{Pdef}
 P_2:=\exp \p{\frac B {M \delta}}, \qquad P_3:=\exp\p{\frac {B(M+1)} {M \delta}}. 
\end{gather} 
Now assume that $0<\delta_0 \leq \delta_1$ is sufficiently small so that  if $0 < \delta \leq \delta_0$ then \begin{gather} \label{uio}
\max_{s \in K} \max_{1 \leq k \leq n} 
\sum_{p \leq P_1}  \abs{\log\p{1-\frac{\omega(p)\chi_k(p)}p}   - \log \p{1- \frac{\omega(p)\chi_k(p)}{ p^{1+\delta s} }}}  <\frac{\varepsilon} 9,
\end{gather}
and such that $P_2$ defined by \eqref{Pdef} is sufficiently large so that $P_2 \geq P_1$ and that we may for each $m=1,\ldots,M$ apply Lemma  \ref{lem5} with  $\xi=\frac 1 m$ and $P=P_2^m =\exp(mB/M \delta)$ so that
\begin{gather} \label{ia3}
  \max_{1\leq k \leq n}\abs{ \sum_{mB/M <  \delta \log p \leq (m+1)B/M}\frac{\omega(p) \chi_k(p)} p -g\p{\frac{mB}M} \frac B M} <\frac{\varepsilon}{9M},
\end{gather}
for some $|\omega(p)|=1$ defined for $mB/M \leq \delta \log p<(m+1)B/M$. 
By   the inequalities  \eqref{ia0}, \eqref{ia1} and \eqref{ia3} for each $m=1,\ldots,M$ and the triangle inequality it follows that
\begin{gather} \label{ia4}
         \max_{s \in K} \max_{1\leq k \leq n}\abs{\sum_{P_2 <   p \leq P_3}\frac{\omega(p)\chi_k(p)} {p^{1+\delta s}} -f_k(s)} <\frac{3\varepsilon}{9}.
\end{gather}
By defining  $\omega(p):=\omega_0(p)$ when $P_1<p\leq P_2$ and when $p>P_3$  it follows that 
\begin{gather} \label{ia5}
       \max_{s \in K} \max_{1\leq k \leq n}\abs{\sum_{P_1 <   p \leq P_2}\frac{\omega(p)\chi_k(p)} {p^{1+\delta s}}} <\frac{2\varepsilon}{9}, \qquad
  \max_{s \in K} \max_{1\leq k \leq n}\abs{\sum_{P_3 <   p}\frac{\omega(p)\chi_k(p)} {p^{1+\delta s}}} <\frac{\varepsilon}{9},
\end{gather}
         by applying the inequality \eqref{oo3}  twice (combined with the triangle inequality) and once respectively. The conclusion of our lemma follows by the inequalities \eqref{iia0}, \eqref{ajaj}, \eqref{uio},  \eqref{ia4}, \eqref{ia5} and the triangle inequality.
\qed

\section{The irrational case - Proof of Lemma \ref{LA} \label{sec5}}

The next lemma is a slightly sharper version of Cassels \cite[Lemma, p. 177]{Cassels}.
\begin{lem} \label{cassel} Let $\alpha>0$ be an algebraic irrational number and  
  let
 $$
   A=\{n \in \Z^+: \log(n+\alpha) \not \in \operatorname{Span}_\Q \{\log(k+\alpha):0\leq k \leq n-1\} \}.
 $$
 Then there exists  some positive sequence  $M_N$  with $M_N=o(N)$ such that for each $N \in \Z^+$ at least $51 \%$ of the integers $n$ in the interval $N<n \leq N+M_N$ belong to the set $A$. 
 \end{lem}

\begin{proof} Cassels use $M=M_N=10^{-6}N$ for a sufficiently large $N$, but it is clear that his proof method applies also for the intervals $[N,N+M_N]$ for some $M_N=o(N)$. In particular it follows from the arguments in \cite[p. 183]{Cassels} that we may choose $M_N= N (\log N)^{-\xi}$ for $0<\xi<1$ and $N$ sufficiently large. \end{proof}

\noindent {\em Proof of Lemma \ref{LA}.}
By Lemma \ref{LA3} we have that for any $B\geq B_0$ there
there exist some function $g \in C[0,B]$  such that
\begin{gather} \label{ii} \max_{s \in K} \abs{\int_{0}^{B} g(x) e^{-sx} dx - f(s)}<\frac{\varepsilon} 3.
\end{gather}
Define $\omega(\alpha):=1$ and for a given $\delta>0$ define 
\begin{gather}  \label{omegadef}
 \omega(n+\alpha):=\exp\left(i \arg  \left( \int_0^{\delta \log(n+\alpha)}  g(x)dx -\sum_{k=0}^{n-1}  \frac{\omega(k+\alpha)} {k+\alpha}  \right) \right),
\end{gather}
recursively for the integers $n \geq 1$  such that \begin{gather} \label{cond} \log(n+\alpha) \not \in \operatorname {Span}_\Q\{\log(k+\alpha):0 \leq	k \leq n-1\},
\end{gather}
and define $\omega(n+\alpha)$ such that  \eqref{cond2}, \eqref{cond3} is satisfied otherwise. If \eqref{cond} does not hold then $\alpha$ must be algebraic, and then \eqref{cond2} must hold for some integers $b_k$ with $b_n  \neq 0$, and $\omega(n+\alpha)$ is defined by \eqref{cond3}.

For any $\varepsilon, \varepsilon_2>0$ it follows by Lemma \ref{cassel} that if $\delta>0$ is sufficiently small then for $N_1:=\exp(\delta^{-\frac 1 2})$ we have
\begin{gather}
  \label{h1}
   \max_{s \in K} \abs{\sum_{0 \leq n< N_1}    \frac{\omega(n+\alpha)} {(n+\alpha)^{1+\delta s}}} <\frac \varepsilon 3,  \\ \intertext{and}  
 \label{iii}
   \abs{\sum_{N_1 \leq n< \exp(X\delta^{-1})}  \frac{\omega(n+\alpha)} {n+\alpha} - \int_{\sqrt \delta}^X g(x)dx}<\varepsilon_2,   \qquad (\sqrt \delta \leq  X \leq B). 
\end{gather}
From the inequality \eqref{iii} and partial summation/partial integration we have since $K$ is compact that
\begin{gather}  \label{iv}
  \max_{s \in K}  \abs{\sum_{N_1\leq n <N}  \frac{\omega(n+\alpha)} {(n+\alpha)^{1+\delta s}} - \int_{\sqrt \delta}^B g(x) e^{-sx} dx}<\frac \varepsilon  3, 
\end{gather}
provided $\varepsilon_2$ has been chosen sufficiently small. Finally our lemma follows
with
\begin{gather*}
N:=\exp(B \delta^{-1})
\end{gather*}
from the inequalities  \eqref{ii}, \eqref{h1},  \eqref{iv} and the triangle inequality.  Since this is construction works for any $B \geq B_0$ we may thus choose $N_0=\exp(B_0 \delta^{-1})$. \qed

\section{Proof of Theorem \ref{th5}}
The proof of Theorem \ref{th5} is identical to the algebraic irrational case of the proof of  Theorem \ref{th}, but we need to replace Lemma \ref{LA} by
\begin{lem} 
 For any $\lambda \in \R$, irrational parameter $\alpha>0$, compact set $K$  and function  $f$,  satisfying the conditions of Theorem \ref{th} there exists  for any  $\varepsilon>0$ some  $\delta_0>0$ such that for any $0<\delta \leq \delta_0$ there exists  some $N_0$ such that for any $N \geq N_0$ there exists some completely multiplicative unimodular function $\omega:\{\alpha,\ldots,N+\alpha\} \to \C$ such that
 \begin{gather} \notag
    \max_{s \in K} \abs{\sum_{n=0}^N  \frac{\omega(n+\alpha)\exp \p{2 \pi i \lambda n}} {(n+\alpha)^{1+\delta s}} -f(s)} <\varepsilon.
\end{gather}
 \end{lem}
 The proof of Lemma \ref{LA3} is identical to the proof of Lemma \ref{LA}, except that we replace \eqref{omegadef} by
\begin{gather*}  
 \omega(n+\alpha):=\exp \p{-2 \pi i \lambda n + i \arg  \left( \int_0^{\delta \log(n+\alpha)}  g(x)dx -\sum_{k=0}^{n-1}  \frac{\omega(k+\alpha)\exp \p{2 \pi i \lambda k}} {k+\alpha}  \right) },
\end{gather*}
and replace $\omega(n+\alpha)$ by $\omega(n+\alpha)\exp \p{2 \pi i \lambda n}$ in the inequalities \eqref{h1}, \eqref{iii}, \eqref{iv}. \qed

\bibliographystyle{plain}
 \end{document}